\documentclass{amsart}
\usepackage{tikz}
\usepackage{xcolor}
\usepackage{amssymb,latexsym,amsmath,extarrows}
\usepackage{graphicx,mathrsfs,comment}
\usepackage{hyperref,url}

\usepackage{amstext}
\usepackage{bbm}

\numberwithin{equation}{section}

\newtheorem{theorem}{Theorem}[section]
\newtheorem{lemma}[theorem]{Lemma}

\newtheorem{proposition}[theorem]{Proposition}

\newtheorem{conjecture}[theorem]{Conjecture}

\newcommand{\R}{\mathbb{R}}

\newcommand{\N}{\mathcal{N}}

\begin{document}

\title[]{On decoupling and restriction estimates}


\author{Changkeun Oh}\address{ Changkeun Oh\\Department of Mathematical Sciences and RIM, Seoul National University, Republic of Korea} \email{changkeun@snu.ac.kr}


\begin{abstract}
In this short note, we prove that the restriction conjecture for the (hyperbolic) paraboloid in $\R^{d}$ implies the $l^p$-decoupling theorem for the (hyperbolic) paraboloid in $\R^{2d-1}$. In particular, this gives a simple proof of the $l^p$ decoupling theorem for the (hyperbolic) paraboloid in $\R^3$.
\end{abstract}

\maketitle



\section{Introduction}

Let $d \geq 2$ and $\vec{v}=(v_1,\ldots,v_{d-1}) \in \{-1,1 \}^{d-1}$. Consider
\begin{equation}
    \mathcal{H}_{\vec{v}}^{d-1}:=\big\{ (\xi_1,\dots,\xi_{d-1},\sum_{i=1}^{d-1}v_i \xi_i^2 ) : \xi_i \in [0,1]^{d-1}  \big\}.
\end{equation}
Define the extension operator for $\mathcal{H}_{\vec{v}}^{d-1}$ by
\begin{equation}
    \widehat{f d\sigma}(x):=\int_{[0,1]^{d-1}}f(\xi)e\big(( x \cdot (\xi_1,\dots,\xi_{d-1},\sum_{i=1}^{d-1}v_i \xi_i^2 ) \big)\, d\xi,
\end{equation}
where $e(t)=e^{2\pi i t}$.
Let us state the well known conjecture in restriction theory.
\begin{conjecture}[Restriction conjecture, \cite{MR545235}]\label{0224.thm11} For $p>2+\frac{2}{d-1}$, we have
\begin{equation}
    \|\widehat{f d\sigma}\|_{L^p(\R^d)} \leq C_p\|f\|_{L^p([0,1]^{d-1})}.
\end{equation}
\end{conjecture}
The restriction conjecture is introduced by Elias Stein. The conjecture is solved when $n=2$ by Fefferman \cite{MR257819}, but it is still open for higher dimensions. For the case $v_1=\ldots=v_{d-1}$, we refer to \cite{wang2024restrictionestimatesusingdecoupling} for the best known result. For the general case 
$\vec{v} \in \{-1,1\}^{d-1}$, we refer to \cite{MR4405679} and references therein.
\medskip

The purpose of this note is to relate a restriction conjecture and a decoupling theorem. A decoupling inequality is first introduced by \cite{MR1800068}. The $l^2$ decoupling inequality for the paraboloid is proved by \cite{MR3374964} and the $l^p$ decoupling inequality for the hyperbolic paraboloid is proved by \cite{MR3736493}. Decoupling inequalities have a number of applications. We refer to \cite{MR4680276} for references. 
\medskip

Let us introduce some notations to state a decoupling inequality. Denote by $\N_{\mathcal{H}_{\vec{v}}^{d-1}}(R^{-1})$ the $R^{-1}$-neighborhood of the hyperbolic paraboloid $\mathcal{H}_{\vec{v}}^{d-1}$.
Given a parameter $\vec{\alpha}=(\alpha_1,\ldots,\alpha_{d-1}) \in [0,1]^{d-1}$, let
$\Gamma_{\vec{\alpha} }(R^{-1})$ be the collection defined by
\begin{equation*}
\begin{split}
    \Big\{ (B \times \R) &\cap \N_{\mathcal{H}_{\vec{v}}^{d-1}}(R^{-1}):  
    \\& B= \big( (c_1,\ldots,c_{d-1})+ [0,R^{-\alpha_1}] \times \cdots \times [0,R^{-\alpha_{n-1}}] \big), \, c_i \in R^{-\alpha_i}\mathbb{Z} \Big\}.
\end{split}
\end{equation*}
For each $B \subset \R^d$ and each $F:\R^d \rightarrow \mathbb{C}$ we define
\begin{equation}
    \mathcal{P}_{B}F(x):= \int_B \widehat{F}(\xi)e(\xi \cdot x)\, d\xi.
\end{equation}
Let us state the $l^p$ decoupling theorem for the hyperbolic paraboloid $\mathcal{H}_{\vec{v}}^{d-1}$ in $\R^d$.
\begin{theorem}[\cite{MR3736493}]\label{0224.thm12} Let $\vec{\alpha}=(\frac12,\ldots,\frac12)$.
    For $2 \leq p \leq 2+\frac{4}{d-1}$ and $\epsilon>0$,
    \begin{equation}
         \|F\|_{L^p(\R^d)} \leq C_{p,\epsilon} R^{\frac12(d-1)(\frac12-\frac1p)+\epsilon} \big( \sum_{\theta \in {\Gamma}_{\vec{\alpha}}(R^{-1}) }\|F_{\theta}\|_{L^p}^p \big)^{\frac1p}
    \end{equation}
         for all functions $F: \R^d \rightarrow \mathbb{C}$ whose Fourier transforms are supported on $\N_{\mathcal{H}_{\vec{\alpha}}^{d-1}}(R^{-1})$.
\end{theorem}

Our theorem is as follows.

\begin{theorem}\label{0224.thm13}

Let $d \geq 2$.
    Conjecture \ref{0224.thm11} for the (hyperbolic) paraboloid in $\R^d$ implies Theorem \ref{0224.thm12} for the (hyperbolic) paraboloid in $\R^{2d-1}$.
\end{theorem}

Note that the critical exponents of $p$ of the restriction conjecture for the hyperbolic paraboloid in $\R^d$ and the decoupling theorem for the hyperbolic paraboloid in $\R^{2d-1}$ are equal.
As mentioned before, Conjecture \ref{0224.thm11} is proved by Fefferman for the case $d=2$. Hence, our theorem gives an alternative proof of the $l^p$ decoupling theorem for the (hyperbolic) paraboloid in $\R^3$.

\subsection{Acknowledgements}
The author was supported by the NSF grant DMS-1800274 and POSCO Science Fellowship of POSCO TJ Park Foundation. The author would like to thank Shaoming Guo for valuable discussions. The author would like to thank a referee for valuable suggestions on writing.

\section{Proof of Theorem \ref{0224.thm13}}

Let us begin with the equivalent formulation of the restriction conjecture.

\begin{proposition}[Proposition 1.5 of \cite{guth2024smallcapdecouplingparaboloid}]\label{0224.prop21}
Let $d \geq 2$ and $\vec{\alpha}=(1,\ldots,1)$.
    If Conjecture \ref{0224.thm11} is true for some $d$, then for $2 \leq p \leq 2+\frac{2}{d-1}$ and $\epsilon>0$, we have
    \begin{equation}
    \|F\|_{L^p(\R^d)} \leq C_{p,\epsilon} R^{(d-1)(\frac12-\frac1p)+\epsilon} \Big( \sum_{\gamma \in \Gamma_{\vec{\alpha}}(R^{-1}) }\| \mathcal{P}_{\gamma}F \|_{L^p(\R^d)}^p \Big)^{\frac1p}
\end{equation}
     for all functions $F: \R^d \rightarrow \mathbb{C}$ whose Fourier transforms are supported on $\N_{\mathcal{H}_{\vec{\alpha}}^{d-1}}(R^{-1})$.
\end{proposition}

The proof of Theorem \ref{0224.thm13} relies on an iterative argument. This argument is inspired by \cite{MR3749372}. Consider the hyperbolic paraboloid $\mathcal{H}_{\vec{v}}^{2d-2}$ in $\R^{2d-1}$ associated with $\vec{v}=(v_1,\ldots,v_{2d-2}) \in \{-1,1\}^{2d-2}$. 
Our goal is to prove that for $ 2 \leq p \leq 2+ \frac{4}{2d-2}$, $\vec{\alpha}=(\frac12,\ldots,\frac12)$, and $\epsilon>0$,
\begin{equation}
         \|F\|_{L^p(\R^{2d-1})} \leq C_{p,\epsilon} R^{\frac12(2d-2)(\frac12-\frac1p)+\epsilon} \big( \sum_{\theta \in {\Gamma}_{\vec{\alpha}}(R^{-1}) }\|F_{\theta}\|_{L^p(\R^{2d-1})}^p \big)^{\frac1p}
    \end{equation}
         for all functions $F: \R^{2d-1} \rightarrow \mathbb{C}$ whose Fourier supports are in $\N_{\mathcal{H}_{\vec{\alpha}}^{2d-2}}(R^{-1})$. 
          For the rest of the proof, assume that $2 \leq p \leq 2+\frac{4}{2d-2}$.
\begin{lemma}\label{0224.lem22}
Suppose that $0 \leq a_2 \leq a_1 \leq \frac12$.
Define two vectors
\begin{equation}
    \begin{split}
        &\vec{\alpha}:=(a_1,\ldots,a_1,a_2,\ldots,a_2) \in [0,1]^{d-1} \times [0,1]^{d-1}
        \\&\vec{\beta}:=(a_1,\ldots,a_1,2a_1-a_2,\ldots,2a_1-a_2) \in [0,1]^{d-1} \times [0,1]^{d-1}.
    \end{split}
\end{equation} 
Under the assumption of Conjecture \ref{0224.thm11} for the hyperbolic paraboloid in $\R^d$ for some $d$, for  given $\tau \in \Gamma_{ \vec{\alpha} }(R^{-1})$ and $\epsilon'>0$, we have
\begin{equation}\label{0224.lem225}
         \|F\|_{L^p(\R^{2d-1})} \leq C_{p,\epsilon'} R^{{(2a_1-2a_2)(d-1)}(\frac12-\frac1p)+\epsilon'} \big( \sum_{\theta \in {\Gamma}_{\vec{\beta}}(R^{-1}) }\|\mathcal{P}_{\theta}F\|_{L^p(\R^{2d-1})}^p \big)^{\frac1p}
    \end{equation}
    for all functions $F$ whose Fourier supports are contained in $\tau$. The constant $C_{p,\epsilon'}$ is independent of the choice of $\tau, a_1, a_2$.
\end{lemma} 

Let us assume the lemma and continue the proof.  Fix $\epsilon>0$. By the assumption on the Fourier support of $F$, we may write
\begin{equation}
    F=\sum_{\theta_1 \in {\Gamma}_{\vec{\alpha_1}}(R^{-1}) }\mathcal{P}_{\theta_1}F
\end{equation}
where $\vec{\alpha_1}:=({\epsilon},\ldots,{\epsilon},0,\ldots,0) \in [0,1]^{d-1} \times [0,1]^{d-1}$.
Since the cardinality of the collection ${\Gamma}_{\vec{\alpha_1}}(R^{-1}) $ is comparable to $R^{{(d-1)\epsilon} }$, by the triangle inequality and H\"{o}lder's inequality, we have
\begin{equation}\label{25.03.07.26}
    \|F\|_{L^p(\R^{2d-1})} \leq C_p R^{(d-1)\epsilon (1-\frac1p) } \big( \sum_{\theta_1 \in {\Gamma}_{\vec{\alpha_1}}(R^{-1}) }\|\mathcal{P}_{\theta_1}F\|_{L^p(\R^{2d-1})}^p \big)^{\frac1p}.
\end{equation}
To simplify the notation, let us introduce  $C:=(d-1)(1-\frac1p)$.

Define
    \begin{equation}
        \vec{\alpha}_n:=
        \begin{cases}
            {\epsilon}( n-1  ,\ldots,n-1, n,\ldots, n )  \;\;\; \mathrm{if} \, n \,
            \mathrm{is \, even}
            \\
            {\epsilon}( n  ,\ldots,n, n-1,\ldots, n-1 ) \;\;\; \mathrm{if} \, n \,
            \mathrm{is \, odd}
        \end{cases}
    \end{equation}
By applying Lemma \ref{0224.lem22} to the right hand side of \eqref{25.03.07.26} with $\epsilon'=\epsilon^2$, $a_1=\epsilon$, and $a_2=0$, we have
\begin{equation*}
    \|F\|_{L^p(\R^{2d-1})} \leq C_{p,\epsilon^2}C_p R^{C\epsilon} R^{2\epsilon(d-1)(\frac12-\frac1p)+\epsilon^2}\big( \sum_{\theta_2 \in {\Gamma}_{\vec{\alpha_2}}(R^{-1}) }\|\mathcal{P}_{\theta_2}F\|_{L^p(\R^{2d-1})}^p \big)^{\frac1p}.
\end{equation*}
Note that we can apply Lemma \ref{0224.lem22} again to the right hand side of the above inequality with $\epsilon'=\epsilon^2$, $a_1=2\epsilon$, and $a_2=\epsilon$ (when we apply the lemma, we  switch $\xi_{d},\ldots,\xi_{2d-2}$ variables and $\xi_1,\ldots,\xi_{d-1}$ variables). By repeating this process $[\epsilon^{-1}/2]$-times, we have 
\begin{equation*}
\begin{split}
    \|F\|_{L^p(\R^{2d-1})} \leq (C_{p,\epsilon^2})^{[\epsilon^{-1}/2]}C_p R^{C\epsilon} &R^{2\epsilon [\epsilon^{-1}/2](d-1)(\frac12-\frac1p)+\frac{\epsilon}{2}} \\& \times \big( \sum_{\theta \in {\Gamma}_{\vec{\alpha}_{[\epsilon^{-1}/2]}}(R^{-1}) }\|\mathcal{P}_{\theta }F\|_{L^p(\R^{2d-1})}^p \big)^{\frac1p}.
\end{split}
\end{equation*}
By the triangle inequality, we finally obtain
\begin{equation}
    \|F\|_{L^p(\R^{2d-1})} \leq \widetilde{C}_{p,\epsilon} R^{2C\epsilon+\frac{\epsilon}{2}} (R)^{(d-1)(\frac12-\frac1p)}\big( \sum_{\theta \in {\Gamma}_{\vec{\alpha}}(R^{-1}) }\|\mathcal{P}_{\theta }F\|_{L^p(\R^{2d-1})}^p \big)^{\frac1p},
\end{equation}
where $\vec{\alpha}=(\frac12,\ldots,\frac12)$ and $\widetilde{C}_{p,\epsilon}=(C_{p,\epsilon^2})^{[\epsilon^{-1}/2]}C_p$.
This is the desired inequality. It remains to prove the lemma.

\begin{proof}[Proof of Lemma \ref{0224.lem22}]
Let us fix $\tau \in \Gamma_{\vec{\alpha}}(R^{-1})$ and suppose that the Fourier support of $F$ is contained in $\tau$. We need to prove \eqref{0224.lem225}. We write
\begin{equation}
    \tau=  \big(\big((c_1,\ldots,c_{2d-2})+ ([0,R^{-a_1}]^{n-1} \times [0,R^{-a_{2}}]^{n-1}) \big) \times \mathbb{R} \big) \cap \N_{\mathcal{H}_{\vec{v}}^{2d-2}}(R^{-1}).
\end{equation}
To simplify the notation, we introduce
\begin{equation}
    \begin{split}
        &c:=(c_1,\ldots,c_{2d-2}), \;\;  \;\; \xi=(\xi',\xi_{2d-1}) \in \R^{2d-2} \times \R,
        \\&
P(\xi'):=v_1\xi_1^2+\cdots+v_{2d-2}\xi_{2d-2}^2.
    \end{split}
\end{equation}
Define the affine transformation
\begin{equation}
    L(\xi):=(\xi'+c,\xi_{2d-1}+P(c) + \xi' \cdot \nabla P(c)), \;\;\;\;\; \xi \in\R^{2d-1}
\end{equation}
and the function $G_0$ satisfying
$    \hat{G}_0(\xi):=\hat{F}(L\xi)$. Then the Fourier support of $G_0$ is contained in the following set
    \begin{equation}
     \big( [0,R^{-a_1}]^{d-1} \times [0,R^{-a_2}]^{d-1} \times \R \big) \cap   \N_{\mathcal{H}_{\vec{v}}^{2d-2}}(16R^{-1}).
    \end{equation}
    Recall that $a_2 \leq a_1$. Define the linear mapping $L_1$ and the function $G$ so that
    \begin{equation}\label{25.03.10.213}
        L_1(\xi):=(R^{-a_2}\xi',R^{-2a_2}\xi_{2d-1}), \;\;\; \hat{G}(\xi):=\hat{G_0}(L_1(\xi))=\hat{F}(L_1L(\xi)).
    \end{equation}
    Then the Fourier support of $G$ is contained in
    \begin{equation}\label{02.24.set}
     \big( [0,R^{a_2-a_1}]^{d-1} \times [0,1]^{d-1} \times \R \big) \cap   \N_{\mathcal{H}_{\vec{v}}^{2d-2}}(16R^{2a_2-1}).
    \end{equation}
    Note that $R^{2a_2-1} \leq R^{-2a_1+2a_2}$.
    By performing the linear  change of variables \eqref{25.03.10.213},  the desired inequality \eqref{0224.lem225} follows from
\begin{equation*}
      \|G\|_{L^p(\R^{2d-1})} \leq C_{p,\epsilon'} {R^{(2a_1-2a_2)(d-1){(\frac12-\frac1p)+\epsilon'}}} \big( \sum_{\theta \in {\Gamma}_{\vec{\beta_0}}(\widetilde{R}^{-1}) }\|\mathcal{P}_{\theta}G\|_{L^p(\R^{2d-1})}^p \big)^{\frac1p}
\end{equation*}
where  $\widetilde{R}^{-1}:=16R^{-2a_1+2a_2}$ and $\vec{\beta_0}:=(
\frac12,\ldots,\frac12,1,\ldots,1) \in \R^{d-1} \times \R^{d-1}$.
To prove the above inequality, let us first  observe that the set \eqref{02.24.set} is contained in
\begin{equation}\label{25.03.10.215}
     [0,R^{a_2-a_1}]^{d-1} \times \mathcal{N}_{\mathcal{H}_{\vec{v_0}}^{d-1}}(d\widetilde{R}^{-1})
\end{equation}
where $\vec{v}_0=(v_d,\ldots,v_{2d-2})$. To prove this, suppose that $\xi=(\xi_1,\ldots,\xi_{2d-1})$ is an element of \eqref{02.24.set}. We need to verify that
\begin{equation}
    \Big|\xi_{2d-1}-\sum_{i=d}^{2d-2}v_i\xi_i^2\Big| \leq d\widetilde{R}^{-1}.
\end{equation}
This is true because 
\begin{equation}
\begin{split}
    &\Big|\sum_{i=1}^{2d-2}v_i\xi_i^2- \sum_{i=d}^{2d-2}v_i\xi_i^2\Big| \leq \sum_{i=1}^{d-1}|\xi_i|^2 \leq  (d-1)R^{-2a_1+2a_2},
    \\&
    \Big| \xi_{2d-1} - \sum_{i=1}^{2d-2}v_i\xi_i^2 \Big| \leq 10R^{2a_2-1} \leq  10R^{-2a_1+2a_2}.
\end{split}
\end{equation} 
Since the Fourier support of $G$ is contained in \eqref{25.03.10.215}, to get the desired inequality, it suffices to decompose frequencies for $\xi_d,\ldots,\xi_{2d-2}$-variables. We apply Proposition \ref{0224.prop21} for the last $d$-variables with $R^{-1}$ replaced by $d\widetilde{R}^{-1}$, and obtain 
\begin{equation*}
      \|G\|_{L^p(\R^{2d-1})} \leq C_{p,\epsilon'} {R^{(2a_1-2a_2)(d-1){(\frac12-\frac1p)+\epsilon'}}} \big( \sum_{\theta \in {\Gamma}_{\vec{\beta_0}}(d\widetilde{R}^{-1}) }\|\mathcal{P}_{\theta}G\|_{L^p(\R^{2d-1})}^p \big)^{\frac1p}.
\end{equation*}
It suffices to apply the triangle inequality and obtain
\begin{equation*}
      \|G\|_{L^p(\R^{2d-1})} \leq {C}_{p,\epsilon'}' {R^{(2a_1-2a_2)(d-1){(\frac12-\frac1p)+\epsilon'}}} \big( \sum_{\theta \in {\Gamma}_{\vec{\beta_0}}(\widetilde{R}^{-1}) }\|\mathcal{P}_{\theta}G\|_{L^p(\R^{2d-1})}^p \big)^{\frac1p}.
\end{equation*}
This completes the proof.
\end{proof}

\bibliographystyle{alpha}
\bibliography{reference}

\end{document}